 \documentclass[12pt]{article} 
  
\usepackage{amsmath}
\usepackage{amsfonts}
\usepackage{amsthm}
\usepackage{amssymb}  
\usepackage{latexsym}
\usepackage{euscript} 
\usepackage{txfonts}  

\newcommand{\msf}[1]{\mathsf {#1}}

\newcommand{\mcal}[1]{{\mathcal {#1}}} 

\newtheorem{theorem}{Theorem}   [section]
\newtheorem{lemma}[theorem]{Lemma}
\newtheorem{corollary}[theorem]{Corollary}
\newtheorem{proposition}[theorem]{Proposition}

\theoremstyle{definition}
\theoremstyle{definition}
\theoremstyle{definition}\newtheorem{example}[theorem]{Example}
\theoremstyle{definition}

\theoremstyle{definition}
\theoremstyle{definition}

\theoremstyle{plain}\newtheorem*{theorem*}{Theorem}
\theoremstyle{plain}\newtheorem*{corollary*}{Corollary}
\theoremstyle{remark}\newtheorem*{remark*}{Remark}
\theoremstyle{remark}\newtheorem*{remarks*}{Remarks}
\theoremstyle{definition}\newtheorem*{conjecture*}{Conjecture}
\theoremstyle{definition}\newtheorem*{definition*}{Definition}
\theoremstyle{definition}\newtheorem*{example*}{Example}

\theoremstyle{definition}\newtheorem*{question*}{Question}
\theoremstyle{definition}\newtheorem*{questions*}{Questions}

\theoremstyle{definition}\newtheorem*{hypothesis*}{Hypothesis}

\def\qed{\ifhmode\unskip\nobreak\fi\ifmmode\ifinner\else\hskip5pt\fi\fi
 \hfill\hbox{\hskip5pt\vrule width4pt height6pt depth1.5pt\hskip1pt}}
 
\newcommand{\ov}[1]{\ensuremath{\overline{#1}}}

\newcommand{\sgn}{\ensuremath{\operatorname{\mathsf{sign}}}}

\newcommand{\RR}{\ensuremath{{\mathbb R}}}     
\newcommand{\R}[1]{\ensuremath{{\mathbb R}^{#1}}} 

\newcommand{\Np}{\ensuremath{{\mathbb N}_{+}}}  
\renewcommand{\S}[1]{\ensuremath{{\bf S}^{#1}}} 

\newcommand{\om}[1]{\ensuremath{\omega(#1)}}
\newcommand{\emp}{\ensuremath{\emptyset}}

\newcommand {\pd}[1] {\ensuremath{\frac{\partial}{\partial #1}}}

\def\d#1dt{\frac{d#1}{dt}}    

\newcommand{\Lam}{\ensuremath{\Lambda}}

\newcommand{\Gam}{\ensuremath{\Gamma}}

\newcommand{\co}{\colon\thinspace} 

\renewcommand{\gg}{\ensuremath{\mathfrak g}}
\renewcommand{\aa}{\ensuremath{\mathfrak a}}

\newcommand{\hh}{\ensuremath{\mathfrak h}}

\newcommand{\p}{\ensuremath{\partial}}

\def\emp{\varnothing}
\reversemarginpar   

	\def\mylabel#1{\label{#1}}
\newcommand{\V}{\ensuremath{{\mcal V}}}
\newcommand{\Z}[1]{\ensuremath{{\msf  Z ( #1)}}}

\renewcommand{\om}{\omega}

\begin{document}

\title{\bf Common zeroes of families of smooth vector fields on surfaces}
\author{{\bf Morris W. Hirsch}
  Mathematics Department\\ University of Wisconsin
 at Madison\\ University of California at Berkeley}
\maketitle
%
\begin{abstract} 
Let $Y$ and $X$ denote $C^k$ vector fields on a possibly noncompact
surface  with empty boundary, $1\le k <\infty$.  Say that $Y$ {\em
  tracks} $X$ if the dynamical system it generates locally permutes
integral curves of $X$.  Let $K$ be a locally maximal compact set
of zeroes of $X$.

\smallskip
\noindent
{\bf Theorem.} {Assume the Poincar\'e-Hopf index of $X$ at $K$ is
  nonzero, and the $k$-jet of $X$ at each point of $K$ is nontrivial.
  If $\gg$ is a supersolvable Lie algebra of $C^k$ vector fields that
  track $X$, then the elements of $\gg$ have a common zero in $K$.}

\smallskip
\noindent
Applications are made to attractors and transformation groups.
\end{abstract}

\tableofcontents

\section{Introduction}   \mylabel{sec:intro}
$M$  denotes a metrizable real analytic surface with empty
boundary.  The vector space of vector fields on $M$ is
$\V(M)$,  topologized by uniform convergence on
compact sets. 
The subspace of $C^r$ vector fields is $\V^r (M)$.  Here $r\in \Np$
(the set of positive integers), or $r=\infty$ (infinitely
differentiable), or $r=\om$ (analytic).

$X$ denotes a vector field on $M$, with zero set $\Z X$.  A compact
set $K\subset \Z X$ is a {\em block} of zeroes for $X$, or an {\em
  $X$-block}, if it has a precompact open neighborhood $U\subset M$
such that $ \Z X \cap \ov U=K$.  We say $U$ is {\em isolating} for $X$
and for $(X, K)$.

 The {\em index} of the $X$-block $K$ is the integer $\msf i_K
 (X):=\msf i (X,U)$ defined as the Poincar\'e-Hopf index \cite
 {Poincare85, Hopf25} of any sufficiently close approximation to $X$
 having only finitely many zeroes in $U$.
This number is independent of $U$, and is stable under perturbations
of $X$ : If $Y\in \V M$ is
sufficiently close to $X$ then $U$ is isolating for $Y$, and $\msf
i(Y, U)=\msf i (X, U)$(Proposition \ref{th:stability}).\footnote{
Equivalently: $\msf i (X, U)$ is the intersection number of $X|U$ with
the zero section of the tangent bundle ({\sc C. Bonatti}
\cite{Bonatti92}).  If $X$ is generated by a smooth local flow $\phi$
then $\msf i (X, U)$ equals the fixed-point index $I (\phi_t|U)$ of
{\sc A. Dold} \cite {Dold72} for sufficiently small $t >0$.}

$K$ is {\em essential} if $\msf i_K (X)\ne 0$, which every isolating
neighborhood of $K$ meets $\Z X$.  This powerful condition implies
that if $Y$ is sufficiently close to $X$ then $\Z Y \cap U
\ne\varnothing$.

We say $Y$ {\em tracks} $X$ provided $X, Y\in \V^1 (M)$, and the
corresponding local flows $\Phi^Y=\{\Phi^Y_t\}_{t\in\RR}$ and
$\Phi^X=\{\Phi^X_t\}_{t\in\RR}$ have the following property: For each
$t\in \RR$ the $C^1$ diffeomorphism $\Phi^Y_t\co \mcal D_t \approx
\mcal R_t$ maps  orbits of $X|\mcal D_t$ to  orbits of $X|\mcal
R_t$.  Equivalently: There exists
$f\co M\,\verb=\=\, \Z X\to \RR$
 such that
$X_p\ne 0 \implies [Y, X]= f(p)X_p$ (see {\sc Hirsch},
 \cite[Prop. 2.4]{Hirsch2015}
  or \cite[Prop. 2.3]{Hirsch2013}).\footnote{Article \cite{Hirsch2013}
    is a preliminary version of \cite{Hirsch2015}.}

  If $X$ spans an ideal in a
Lie algebra $\gg\subset \V^k (M)$,  every element of $\gg$ tracks
$X$. When $X$ is $C^\infty$, the set of $Y\in \V^\infty(M)$ that track $X$
is an infinite dimensional  Lie algebra.

\subsection{Statement of the main results}   \mylabel{sec:main}
Throughout the rest of this article we assume:
\begin{itemize}

\item  $X \in \V^k (M)$, \ 
$k\in \Np$.
\end{itemize}

 $X$ has {\em order} $\msf{ord}_p (X):=j\in \{1,\dots,k\}$ at $p\in \Z
X$ if $j$ is the smallest number in $\{1,\dots,k\}$ such that the
$j$-jet of $X$ at $p$ is nontrivial.  In other words: Some (and hence
every) $C^{k+1}$ chart $M\supset W' \approx W\subset \R 2$ on $M$,
centered at $p$, represents $X|W$ by a $C^k$ map $F\co W \to \R 2$
whose partial derivatives at the orgin satisfy
\begin{equation}                \label{eq:ordi}
F^{(i)}(0)=0 \ \text{for $i=0,\dots, j-1$}, \qquad F^{(j)} (0)\ne 0.
\end{equation}
If no integer $j$ has this property, $X$ is {\em $k$-flat} at $p$.

When $X$ is analytic and nontrivial on a neighborhood of a continuum
$L\subset \Z X$, the order of $X$ is constant on $L$.

\begin{theorem}         \mylabel{th:MAIN}
Assume $X, Y\in \V^k (M)$ have the following properties:
\begin{description}

\item[(a)] $K\subset\Z X$ is an essential $X$-block,

\item[(b)] $X$ is not $k$-flat at any point of $K$,

\item[(c)] $Y$ tracks $X$.

\end{description}
Then $\Z Y\cap K\ne\varnothing$.
\end{theorem}
\noindent
When $X$ and $Y$ are commuting analytic vector fields, this is a
special case of a remarkable theorem of {\sc C. Bonatti} \cite{Bonatti92}--- the
inspiration for the present paper. 

Theorem \ref{th:MAIN} is proved by demonstrating the strong form of
the contrapositive stated below.

A {\em line field} $\Lam$ on a set $N\subset M$ is a (continuous)
section $p\mapsto \Lam_p$ of the fibre bundle over $N$ whose fibre
over $p\in N$ is the circle of unoriented lines through the origin in
the tangent space $T_p (M)$.  If $Y_p\in\Lam_p$ for all $p\in N$ then
$\Lam$ {\em controls $Y$ in $N$}.

\begin{theorem}             \mylabel{th:MAINbis}
Assume:
\begin{description}
\item[(a)]   $X\in \V^k (M)$ is not $k$-flat at any point of the
  $X$-block $K$,

 \item[(b)] $Y\in \V^k (M)$ tracks $X$,

\item[(c)] $\Z Y\cap K=\emp$,

\item[(d)] $U\subset M$ is an isolating neighborhood for $(X, K)$.
\end{description}
Then:
\begin{description}

\item[(i)]  $\msf i_K (X) =0$.

\item[(ii)] $K$ has only finitely many components, and each component is
  a $C^k$-embedded circle. 

\item[(iii)] $X$ is controlled by a unique line field in $U$.

\item[(iv)] $X$ has index zero at each component of $K$. 

\item[(v)] $X$ can be
  $C^k$-approximated by vector fields that have no zeroes in $U$ and
  agree with $X$ outside $U$.

\end{description}
\end{theorem}
 The is in Section \ref{sec:proof}.

\subsection{Applications}   \mylabel{sec:apps}


Let $\gg\subset \V^k (M)$ denote a Lie algebra--- a linear subspace
closed under Lie brackets.  The  zero set of $\gg$ is defined as $\Z \gg
:=\bigcap_{Y\in\gg}\Z Y$.  If every $Y\in \gg$ tracks $X$, then {\em
  $\gg$ tracks $X$}. We call $\gg$ {\em supersolvable} if it is
faithfully represented by upper triangular real matrices.

\begin{theorem}         \mylabel{th:liealg}
Assume: 
\begin{description}

\item[(a)]  $K$ is an essential  $X$-block,

\item[(b)]    $X$ is not $k$-flat at any point of $K$,

\item[(c)] $\gg\subset \V^k (M)$ is a supersolvable
  Lie algebra tracking $X$.
\end{description}
Then $\Z\gg \cap K\ne\varnothing$.  
\end{theorem}

 Related theorems and counterexamples are discussed in {\sc Hirsch}
 \cite{Hirsch2015}.

\begin{theorem}         \mylabel{th:app1}
Suppose the local flow of $X\in \V^\om (M)$ has a compact attractor
$P\subset M$ with Euler characteristic $\chi (P)\ne 0$. Then there
exists $k\in \Np$ with the following property: 
If $\hh\subset \V^k
(M)$ is a supersolvable Lie algebra that tracks $X$, then $\Z\hh \cap
\Z X \cap P \ne\varnothing$.
\end{theorem}
\begin{proof} If  $P=M$ then $M$ is a closed surface and $\chi (M)\ne
  0$.  Therefore $\Z X$ is an essential $X$-block by Poincar\'e's
  Theorem \cite{Poincare85}, and  the conclusion follows from
  Theorem \ref{th:MAINbis}.

Suppose $P\ne M$.  The basin of attraction of $P$ contains a smooth
compact surface $N$ with boundary such that $Y$ is inwardly transverse
to $\p N$. The interior $N_0:=N\,\verb=\=\, \p N$ is a precompact open
set which is positively invariant under $\Phi^Y$ ({\sc F. Wilson}
\cite[Th{.\ }2.2]{Wilson69}).  Therefore
\[t > s\ge 0\implies \Phi^Y_t (N_0)\subset  \Phi^Y_s (N_0), \qquad 
\textstyle \bigcap_{t \ge 0}\Phi^Y_t (N_0) = P.
\]
The inclusion maps $P\hookrightarrow N_0\hookrightarrow N$ induce isomorphisms of
\v{C}ech cohomology groups, hence $\chi (N_0)\ne 0$.  The conclusion
follows from Theorem \ref{th:liealg} appplied to $X:= Y|N_0$ and the
Lie algebra $\gg\subset \V^k (N_0)$ comprising the restrictions of the
vector fields in $\hh$ to $N_0$.
\end{proof}

\begin{example}         \mylabel{ex:P}
Assume 
$P\subset \R 2$ is a compact global attractor for $X\in \V^\om (\R 2)$
and $\hh\subset \V^k(\R 2)$ is a 
supersolvable Lie algebra tracking $Y$.  
Then:
\begin{itemize}
\item {\em $\Z\hh \cap \Z X \cap P\ne\varnothing$.}
\end{itemize}
{\em Proof. }  
 $\chi (P)\ne 0$ because $\R 2$ is contractible and $P$ is a global attractor, so 
Theorem \ref{th:app1} yields the conclusion. \qed
 \end{example}

\begin{corollary}               \mylabel{th:app2}
 Let $G$ be a connected Lie group whose Lie algebra is supersolvable,
 with an effective $C^\infty$ action on a compact surface $M$.  If
 $\chi (M)\ne 0$ and the action is analytic on a normal 1-dimensional
 Lie subgroup, then $G$ fixes a point of $M$.
\end{corollary}
\noindent
The special case in which $M$ is compact and $G$ acts analytically
 is due to {\sc Hirsch \& Weinstein} \cite{HW2000}.
\begin{proof}
The action of $G$ on $M$ induces an isomorphism $\theta$ from the Lie
algebra $\aa$ of $G$ onto a subalgebra $\gg\subset \V^\infty (M)$.
Let $Y\in \aa$ span the Lie algebra of $H$ and set $\theta (Y)=X\in
\gg$. 
Then $X\in \V^\om (M)$ and 
$\gg$ tracks $X$, whence the conclusion 
from Theorem \ref{th:app1}.
\end{proof}

Related results on  Lie 
group actions and Lie algebras of vector fields can be found in the articles 
\cite{Belliart97, BL96, Borel56, Hirsch2010, Hirsch2011, Hirsch2014,
  Lima64, Plante86, Plante88, Plante91, Sommese73, Turiel89, Turiel03}. 


\section{Index calculations}   \mylabel{sec:index}

\begin{proposition}             \mylabel{th:invariance}
Assume $Y, X \in \V^k (M)$ and $Y$ tracks $X$.  
\begin{description}

\item[(i)] $\Z X$ is invariant under $\Phi^Y$.

\item[(ii)] If $p, q\in \Z X$ are in the same  orbit of $\Phi^Y$, then
$\msf {ord}_p (X)=\msf {ord}_q (X)$.  \qed

\end{description}
\end{proposition}
\begin{proof} Follows from the definition of tracking.
\end{proof}
These properties of the index function are crucial: 
\begin{proposition}[{\sc  Stability}]   \mylabel{th:stability}
Let $U\subset M$ be  isolating for $X$. 
\begin{description}

\item[(a)] If $\msf i (X, U)\ne 0$ then $\Z X\cap U\ne\varnothing$. 

\item[(b)] If   $Y$ is sufficiently close to $X$
then  $\msf i (Y,  U)=\msf i (X, U)$.

 \item[(c)] Let $\{X^t\}_{t\in [0, 1]}$ be a deformation of $X$.  If
   each $X^t$ is nonsingular on the frontier of $U$, \,then
$\msf i (X^t,  U)=\msf i (X, U).$
\end{description}
\end{proposition}
\begin{proof}
This is Theorem 3.9 of  
\cite{Hirsch2015}.
\end{proof}

\begin{proposition}             \mylabel{th:wedge}
Assume  $Y, Y' \in \V (M)$ and  $U\subset M$ 
is isolating for both $Y$ and $Y'$.  Assume 
$N:=\ov U$ is a compact 
$C^1$ surface  such that  $Y_p$ and  $Y'_p$ are linearly dependent at
all $p\in \p N$. Then $\msf i (Y,U)=\msf i(Y', U)$.
\end{proposition}
\begin{proof}
This consequence of Proposition \ref{th:stability} is a special case
of \cite[Prop. 3.12]{Hirsch2013} or 
 \cite[Prop. 3.11]{Hirsch2015}.  
\end{proof}

\begin{proposition}             \mylabel{th:lines3}
Let $K\subset M$ be a block of zeroes for $X\in \V^k (M)$, and $U\subset
M$ an isolating neighborhood for $(X,K)$.  If $X|U$ is controlled by a
line field $\Lam$ on $U$, then  $\msf i_K (X) =0$.
\end{proposition}
\begin{proof}  By shrinking $U$ slightly, we assume $N:=\ov U$ is a
  compact $C^1$ surface.

Suppose $\Lam$ is an orientable line field.  Then $\Lam$ controls a
nonsingular vector field $Y$ on $N$, and $\msf i (Y,U) =0$ because $\Z
Y=\emp$.  Evidently $U$ is isolating for both $X$ and $Y$, and $X_p$,
$Y_p$ are linearly dependent at all $p\in N$.  Proposition
\ref{th:wedge} implies $\msf i (X, U)=\msf i (Y, U)= 0$, hence $\msf
i_K (X)=0$.

Now suppose $\Lam$ is nonorientable.  There is a double covering
$\pi\co \tilde V\to V$ of an open neighborhood $V\subset M$ of $N$,
isolating for $(X, K)$, such that $\Lam$ lifts to an orientable line
field on $\tilde V$.  The orientable case shows that the vector field
$\tilde X$ on $\tilde V$ that projects to $X|V$ under $\pi$ has index
zero in $\tilde V$.

Fix $X_1\in \V (V)$  with  $\Z{X_1}$ finite, and such that 
the sum of the Poincar\'e-Hopf indices of the zeroes of $X_1$ equals
$\msf i (X, V)$.
 Define $\tilde X_1\in \V (\tilde V)$ to be the vector field
 projecting to $X_1$ under $\pi$.  Each zero $p$ of $X_1$ is the image
 under $\pi$ of exactly two zeroes $q_1, q_2$ of $\tilde X_1$, and the
 Poincar\'e-Hopf indices of $\tilde X$ at $q_1$ and $q_2$ both equal
 the Poincar\'e-Hopf index of $X$ at $p$.  Therefore
\[0=\msf i (\tilde X_1, \tilde V)= 2\msf i (X_1, V) =
 2\msf i (X,V)=2\msf i_K(X),
\]
completing the proof.
\end{proof}

Other calculations of indices can be found in
\cite{Dold72, Gottlieb86, Jubin09, Morse29, Pugh68}.

\section{Proof of Theorem \ref{th:MAINbis} }   \mylabel{sec:proof}
\begin{lemma}           \mylabel{th:lk}
Let $L\subset K$ be a component.  Then $X$ has the same order at each
point of $L$.
\end{lemma}
\begin{proof} The sets  $L_j=\{p\in L\co\msf {ord}_p (X)=j\}$,
  $j\in\{1,\dots,k\}$ are relatively open in $L$ and mutually
  disjoint.  As $L$ is connected and covered by the $L_j$, it
  coincides with  one of them. 
\end{proof}

Let $p\in K$ be arbitrary. 
Choose a $C^{k+1}$ {\em flowbox} $h_p$ for $Y$ centered at $p$:
\begin{equation}                \label{eq:hp}
h_p\co W_p' \approx W_p =J_p\times J'_p \subset \R 2, \quad h (p)= (0,0). 
\end{equation}
This means  $W'_p$ is open in $M$,\,
$J_p, J_p' \subset \RR$ are open intervals around $0$, and the
$C^{k+1}$ diffeomorphism $h_p$ transforms 
$Y|W'_p$ to the constant vector field 
$\left.\pd{x}\right| W_p$,
 where $x, y$ are the usual planar coordinates.  Notice that $h_p (K\cap
 W'_p)=J_p\times \{0\}$ because $K$ is $Y$-invariant.

The transform of $X|W'_p$ by $h_p$ is a $C^k$ vector field
\[\hat X (p)\in \V^k (W_p), \quad  (x,y)\mapsto F_p (x, y),
\]
where $F_p\co W_p\to \R 2$ is $C^k$.  

Set  $\msf {ord}_p (X)=l \in\{1,\dots,k\}$.   The partials of $F_p$ satisfy
\[
F_p^{(i)}(0)=0 \ \text{for $i=0,\dots, l-1$}, \qquad F_p^{(l)} (0)\ne 0.
\]
In a sufficiently small open disk

\begin{equation}                \label{eq:ddp}
D:=D_p\subset \R 2,\quad  (0,0) \in D_p,
\end{equation}
the $l$'th order Taylor expansion of $F_p$ about $(0,0)$ takes the form
\begin{equation}                \label{eq:xxy}
F_p(x,y) =y^l g (x, y), \qquad g(x, y) \ne (0,0) \ \ 
\end{equation}
with $g\co D_p\to \R 2$ continuous.  Therefore 
\[\Z{\hat X(p)}=F_p^{-1} (0,0)= J_p\times\{0\},
\]
 whence $K\cap W'_p$ is an open arc,
relatively   closed in $W'_p$. 

It follows that $K$ has an open cover by open arcs.  Thus $K$ is a
compact $1$-manifold having only finitely many components, each of
which is a topological circle.  The restriction of $\Phi^Y$ to any
component $L\subset K$ is a smooth flow with no fixed points.
Therefore $L$ is a periodic orbit of $\Phi^Y$, and is thus a smooth
submanifold. This proves Theorem \ref{th:MAINbis}(ii).

\begin{lemma}           \mylabel{th:lines0}
$\hat X$ is controlled by a unique line field $\Lam (p)$ on $D_p$.
\end{lemma}
\begin{proof} 
Consider the unit vector field $\hat F$ on $D':=D\,\verb=\=\,J\times\{0\}$, 
as
\[
  \hat F(x,y):= \sgn (y) \frac{F (x,y)}{\|F (x,y)\|} 
\]
where $\|\cdot\|$ denotes the Euclidean norm.
Equation (\ref{eq:xxy}) implies 
\[\mbox{$\displaystyle\lim_{y\to 0}  \hat F(x,y) = \frac {g(x,0)}{\|g (x,0)\|}$
 \  uniformly in 
 $D'$.}
\] 
Therefore $\hat F$ extends to a unique continuous map
$\tilde F\co D \to \S 1$ (the unit circle).
The desired line field sends $(x, y) \in D$ to the line through
$(0,0)$ spanned by $\tilde F_{(x,y)}$. 
\end{proof}

Next we prove \ref{th:MAINbis}(iii).  For each $p\in K$ let $V_p:=
h_p^{-1} (D_p)$, with notation as in Equations (\ref{eq:hp}),
(\ref{eq:ddp}). Define $V:=\bigcap_{p\in K}V_p$.
Let $\Lam (p)$ be the line field defined in  Lemma \ref{th:lines0}. 
The pullback of $\Lam (p)$ by
$h_p$  is the unique line field
$\Gam (p)$ on $V_p$ controlling $X|V_p$.  The unique line field $\Gam$ on $V$
that restricts to $\Gam (p)$ for each $p\in K$ has the required properties.

Parts (i) and (iv) of \ref{th:MAINbis} are consequences of (iii) and 
Proposition  \ref{th:lines3}.  

Part (v) follows from Propositions 3.13 and 3.14 of  \cite
{Hirsch2015} when $U$ is connected, and this implies the 
general case because the compact set $K$ is covered by finitely many
components of $U$.


\end{document}